\documentclass[12pt,a4paper]{article}

\usepackage[utf8]{inputenc}
\usepackage[english]{babel}
\usepackage[T1]{fontenc}
\usepackage{amsmath,amssymb,amsthm}
\usepackage{graphicx}
\usepackage[margin=2.5cm]{geometry}
\usepackage{hyperref}
\usepackage{natbib,setspace}
\usepackage{algorithm,comment}
\usepackage{algorithmic}
\usepackage{booktabs}
\usepackage{xcolor,url} % for writting in colour
\usepackage[normalem]{ulem} % for striking out text
\usepackage{float}
\usepackage{enumitem}
\usepackage{multirow,comment}
\usepackage{soul}

\theoremstyle{plain}

\newtheorem{proposition}{Proposition}[section]

\theoremstyle{definition}
\newtheorem{definition}{Definition}[section]

\hypersetup{
    colorlinks=true,
    linkcolor=blue,
    citecolor=blue,
    urlcolor=blue
}

\definecolor{myblue}{RGB}{0, 60, 160}      
\definecolor{mypurple}{RGB}{100, 0, 160}   
\definecolor{newred}{RGB}{180, 30, 30}
\definecolor{newnewred}{RGB}{180, 0, 120}

\title{Dimensional Decomposition and Column Generation for Bin Dimensioning in E-Commerce Fulfillment Centers}
\author{
    Gabriel González\thanks{Instituto Tecnológico de Aeronáutica (ITA), São José dos Campos, Brazil.}
    \and
    Mariá C. V. Nascimento\thanks{Instituto Tecnológico de Aeronáutica (ITA), São José dos Campos, Brazil.}
}
\date{}

\begin{document}
\maketitle
\onehalfspacing

\begin{abstract}

A fulfillment center (FC) is a specialized logistics facility where e-commerce inventory is received, stored, and processed. Its internal organization directly impacts space utilization and the performance of put-away and picking operations. Items are stored in bins following strict operational guidelines to ensure productivity in both activities, so bin dimensions must be tailored to the product profile of each facility to maximize space utilization.
We address the bin dimensioning problem arising in e-commerce fulfillment centers: given a set of candidate bin types and a large, heterogeneous inventory, determine the number of bins of each type and assign all items so as to minimize total bin volume, subject to operational constraints governing how products may be combined within a single bin. We develop a dimensional decomposition that exploits these constraints to reduce the three-dimensional packing problem to a one-dimensional block-positioning problem. Building on this decomposition, we formulate the problem as a mixed-integer linear program and propose two solution methods: a Best-Fit-Decreasing (BFD) heuristic and a column generation scheme that produces tight LP lower bounds. Computational experiments on four synthetic datasets, calibrated on proprietary data from an American e-commerce fintech and ranging from 1.5 thousand to 1.5 million SKUs, show that the BFD heuristic scales to instances with up to 1.7 million blocks, and that column generation certifies a BFD--LP gap of 2.26\% on the smallest dataset. The experiments further reveal a structural divide in solution quality across datasets, driven by the interaction between operational constraints and each facility's product profile.

\end{abstract}

%==============================================================================
\section{Introduction}
\label{sec:introduction}
%==============================================================================
 
The rapid growth of e-commerce has led companies such as Amazon, Alibaba, and
Mercado Libre to develop increasingly large and complex logistics networks. At
the core of these networks are fulfillment centers---warehouses where products
from thousands of sellers are received, stored, and subsequently picked, packed,
and shipped to customers. These facilities handle an enormous variety of
products: a single center may store millions of individual items spanning
hundreds of thousands of distinct  Stock Keeping Units (SKUs), each of which with
different dimensions, weights, and replenishment profiles. Products are
organized inside storage structures known as shelves and racks, and subdivided into
\emph{bins}, the smallest individually addressable storage compartments. The
dimensions of these bins have a direct and significant impact on warehouse
performance: bins that are too large for the product profile waste space and
reduce storage capacity; bins that are too small may be unable to accommodate
many products or force excessive fragmentation of inventory. Beyond capacity,
bin dimensions also affect picking productivity, since operators must visually locate
products among potentially many cohabitants, and put-away efficiency, since
workers must find appropriate bins for incoming
stock~\citep{gu2007research,gu2010research,dekoster2007design}.
 
For specialized retailers that sell a narrow range of products in large
quantities, bin sizing is straightforward: bins can be tailored to the specific
dimensions of each product line. E-commerce companies face a fundamentally
different challenge. Their catalogs contain a vast diversity of SKUs with very
few units of each in stock, and multiple distinct SKUs typically share the same
bin. This makes the  bin dimensioning problem, in which the aim is to determine which bin types to use and in what proportions, far more complex, since it requires jointly considering the geometric compatibility of thousands of heterogeneous products. This complexity is compounded by operational constraints that govern how these products may be combined within a single bin.

This paper addresses the bin dimensioning problem as presented by an American e-commerce fintech,
one of the largest e-commerce platforms worldwide, whose distribution centers
range from a few thousand to millions of distinct SKUs.
Since the bin catalog must accommodate worst-case storage requirements, the inventory used as input is designed to represent a peak-demand scenario, such as the days surrounding Black Friday, based on statistical characteristics extracted from the company's real data.
Given a catalog of candidate bin types, defined by
their length, width and height, and an inventory instance to be stored, the
goal is to determine the number of bins of each type and assign all items to bins so as to
minimize the total bin volume used, subject to three operational constraints: a
limit on the number of distinct SKUs per bin (to preserve picking productivity),
a limit on the number of units of each SKU per bin (reflecting replenishment
patterns), and a stacking rule that prohibits placing items of different SKUs on
top of or behind one another (to guarantee direct accessibility). This problem is a large-scale variant of the assortment problem, with the shipper-sizing problem studied by \citet{alonso2016determining} as its closest antecedent. It differs from the existing literature in the geometry induced by an operational stacking rule and in its instance scale, which is several orders of magnitude larger.

The scale of real instances, with up to $1.5 \times 10^6$~SKUs and $4 \times
10^6$~items, rules out any direct approach based on three-dimensional geometric
packing, which would require quadratic non-overlap constraints between pairs of
items. A central contribution of this work is a \emph{dimensional decomposition}
that exploits the interaction between the stacking constraint and the structure
of identical items within each SKU to reduce the original three-dimensional
problem to a one-dimensional block-positioning problem. This reduction is non-trivial: it requires aggregating items into
blocks, proving that only the X-axis coordinate matters for non-overlap between
blocks of distinct SKUs, and showing that the optimal X-dimension of each block
can be computed in $O(1)$ time for any given bin type.
Built on this one-dimensional representation, we propose a compact MILP
formulation that captures all operational constraints, and a column generation
scheme with pricing subproblems that yields tight LP lower bounds---certifying,
for instance, a 2.26\% BFD--LP gap on one of our datasets. To handle the
largest instances, we design a Best-Fit-Decreasing (BFD) heuristic that scales
to 1.7~million blocks. We validate the entire methodology on four synthetic datasets generated from proprietary real-world data provided by the company.
This work constitutes the first stage of a broader storage design project; subsequent stages will address the distribution of bins
across the physical shelving structure and the assignment of items to specific
bin positions to minimize operator effort during put-away and picking.

The remainder of this paper is organized as follows.
Section~\ref{sec:related_works} reviews the related literature.
Section~\ref{sec:problem_definition} defines the problem formally.
Section~\ref{sec:decomposition} develops the dimensional decomposition.
Section~\ref{sec:compact_formulation} presents the MILP formulation.
Section~\ref{sec:column_generation} describes the column generation scheme.
Section~\ref{sec:bfd} presents the BFD heuristic.
Section~\ref{sec:experiments} reports the computational experiments.
Section~\ref{sec:conclusions} concludes the paper presenting some final remarks and future work directions.
 
%==============================================================================
\section{Related Works}\label{sec:related_works}
%==============================================================================

The bin dimensioning problem belongs to the family of \emph{assortment problems} and is classified by \citet{wascher2007} as a multiple bin-size bin packing problem. For the two-dimensional case, \citet{beasley1985algorithm} introduced an integer programming model based on a cutting-pattern formulation that remains standard. Subsequent studies developed simulation, genetic, and heuristic approaches for the glass, sheet-metal, and caravan-manufacturing industries \citep{chambers1976cutting,gemmill1990approximate,gemmill1991comparison,gemmill1992solution,agrawal1993determining,holthaus2002methodology,holthaus2003best,arbib2007optimization,arbib2009exact,dowsland2007simulated}.

The closest antecedent is the shipper-sizing problem of \citet{alonso2016determining}, who extend Beasley’s model to a distribution center serving retail outlets and derive lower bounds from multiknapsack, $p$-median, and facility-location relaxations. Their experiments use 20-order samples drawn from historical demand distributions and three to seven product types, with identically dimensioned products aggregated, and evaluate the selected shipper sets on 200-order samples.

The bin packing problem and its variants have a vast literature~\citep{coffman1996approximation,delorme2016bin}. Classical results
for the one-dimensional case include the asymptotic analyses of
First-Fit-Decreasing and Best-Fit-Decreasing
heuristics~\citep{johnson1973near}. 
The variable-sized bin-packing problem, in which bins of different capacities are available, is another name for the same problem class. It was introduced by~\citet{friesen1986variable} and later studied using branch-and-price methods by~\citet{correia2008solving}.
Bin packing with conflicts, where certain items
may not share a bin, has been addressed by~\citet{jansen1999approximation}
and~\citet{epstein2008bin}; the SKU-diversity limit in our problem, which
bounds the number of \emph{distinct} SKUs allowed per bin, can be viewed as a
cardinality-based generalization of such conflict constraints. In higher
dimensions, \citet{martello2000three} proposed the first exact algorithm for
three-dimensional bin packing, and~\citet{pisinger2007using} developed
decomposition-based approaches. Column generation, introduced for the cutting stock problem by
\citet{gilmore1961linear,gilmore1963linear}, whose pattern-based formulation
was independently anticipated by \citet{kantorovich1960mathematical}, as
documented by \citet{uchoa2026kantorovich}, has become a standard tool for
obtaining tight LP bounds in bin packing; see~\citet{vance1998branch}
and~\citet{vanderbeck1999computational} for branch-and-price algorithms,
and~\citet{pessoa2010exact} and~\citet{benamor2006dual} for stabilization
techniques.

%Three features distinguish our problem from that of \citet{alonso2016determining}. First, they select at most four freely dimensioned sizes from over 10{,}000 candidates for a small number of product types, whereas our catalog is predetermined by the company’s manufacturing infrastructure, with $|T| \in \{24,30\}$ and up to $1.5 \times 10^{6}$ SKUs to assign. Second, they sample demand from a probability distribution of future orders, whereas we use a deterministic peak-inventory snapshot. Third, our side-by-side stacking rule enables the exact one-dimensional reduction in Section~\ref{sec:decomposition}, making the pricing subproblem a cardinality-constrained selection over scalar widths; in their setting, it would be a two-dimensional knapsack problem, an approach they consider and reject.

%Rather than benchmarking against earlier assortment-problem studies, \citet{alonso2016determining} evaluate their heuristics using problem-specific lower bounds; we follow the same convention. Their utilization figures are also incommensurable with ours for the reasons discussed in Section~\ref{subsec:fullscale}.

The idea of reducing a multi-dimensional packing problem to a
lower-dimensional one by exploiting structural restrictions on how items may
be arranged recurs throughout the cutting and packing literature. The classical
instance is level- (or shelf-) packing, in which rectangles are forced into
horizontal levels so that a two-dimensional problem decomposes into a sequence
of one-dimensional subproblems~\citep{coffman1980performance,baker1983shelf}.
There, however, the level structure is a self-imposed simplification that
sacrifices optimality. Relaxation- and bound-oriented reductions of the same
flavour also appear: the lower bounds of \citet{martello2000three} and the
decomposition of \citet{pisinger2007using} collapse the geometry to fewer
dimensions to obtain tractable relaxations rather than exact reformulations.

Reductions motivated by operational requirements rather than modeling convenience also appear in the literature. In \citet{alonso2016determining}, a \emph{this-side-up} requirement, whereby product barcodes must remain readable when the lid is opened, fixes one item dimension vertically and thereby yields an exact reduction from three to two dimensions. Methodologically, the work most closely related to our reduction is that of \citet{delorme2024exact}, who decompose a three-dimensional container loading problem with stacking and multi-drop accessibility constraints by first generating item columns and then solving a two-dimensional knapsack problem, thus explicitly bridging two-dimensional packing theory and practical container loading.
Our reduction differs in that the
stacking constraint applies to every pair of distinct SKUs, so the collapse
from three dimensions to one is exact and global---a reformulation of the
original problem rather than a relaxation, a bound, or a heuristic
restriction.

A recurring feature of this literature is the scale at which each method
operates, and a fair comparison must pair instance size with the type of
solution guarantee obtained. Exact methods for two- and three-dimensional bin
packing solve to optimality instances of up to roughly one hundred items---90
items in the three-dimensional branch-and-bound of \citet{martello2000three}
and 100 rectangles in the two-dimensional branch-and-price of
\citet{pisinger2007using}---owing to the cost of enforcing geometric
non-overlap explicitly. One-dimensional variants scale substantially further:
branch-and-price and pseudo-polynomial formulations routinely solve instances
with up to about one thousand items
\citep{correia2008solving,delorme2016bin}. The problem we address originates in
three dimensions but, through the dimensional decomposition proposed in
Section~\ref{sec:decomposition}, reduces to a one-dimensional capacity problem.
This places our column generation at a scale comparable to the one-dimensional
state of the art---certifying LP bounds for instances with up to roughly
2{,}000 SKUs---while the Best-Fit-Decreasing heuristic produces feasible
solutions for instances with up to 1.7~million blocks, several orders of
magnitude beyond the sizes reported for exact geometric packing.

Within warehouse logistics, the \emph{slot profile design problem}---determining
slot heights in pallet racks to maximize space utilization---was studied
by~\citet{cardona2018determine} and extended to full rack layout
by~\citet{cardona2020layouts}. These works consider only one dimension (height)
for single-SKU pallet storage, a setting far simpler than ours. The bin sizing
problem for shipping cartons has been addressed by~\citet{singh2020carton}
and~\citet{chen2023hybrid}; in those problems, however, the item groupings
(customer orders) are given in advance, whereas in our problem the assignment of
SKUs to bins is itself a decision variable.

%==============================================================================
\section{Problem Definition}
\label{sec:problem_definition}
%==============================================================================

We begin by introducing the fundamental concepts needed to formally characterize the problem.

\begin{definition}[SKU]
A SKU (Stock Keeping Unit) is a unique code identifying a specific product together with all its attributes (size, color, weight, etc.). All items sharing the same SKU are physically identical in their dimensions.
\end{definition}

\begin{definition}[Item]
An item is a single physical unit of a product. Each item belongs to exactly one SKU and has the dimensions associated with that SKU.
\end{definition}

\begin{definition}[Bin Type]
A bin type $t \in T$ is characterized by three dimensions: length $L_t$, width $W_t$, and height $H_t$. The set $T$ of available types is fixed in advance according to manufacturing and infrastructure constraints.
\end{definition}

The input data of the problem consist of:
\begin{itemize}
    \item $S$: set of SKUs;
    \item $d_s$: number of items of SKU $s \in S$;
    \item $(l_s, w_s, h_s)$: dimensions (length, width, height) of the items of SKU $s$;
    \item $T$: set of available bin types;
    \item $(L_t, W_t, H_t)$: dimensions of bin type $t \in T$;
    \item $V_t$: volume of bin type $t \in T$, calculated by the product $L_t W_t H_t$;
    \item $m$: maximum number of distinct SKUs allowed in a single bin;
    \item $n_s$: maximum number of items of SKU $s$ allowed in a single bin;
    \item $J$: a set of candidate bins, where for each bin $j\in J$ one must decide whether it is used and, if so, its type $t_j$.
\end{itemize}

The goal is to simultaneously select a type for each required bin $j\in J$ and assign items to bins so as to minimize the total volume of bins used, subject to all operational constraints.

The problem is governed by three operational constraints, each motivated by practical requirements of warehouse storage operations.

\paragraph{SKU Limit per Bin.}
Each bin may contain items of at most $m$ distinct SKUs. This constraint is driven by picking productivity: excessive product diversity within a bin increases the cognitive load on warehouse operators, lengthens visual search time, and raises the probability of retrieval errors.

\paragraph{Quantity Limit per SKU.}
For each SKU $s$, a bin may hold at most $n_s$ units. This parameter captures the replenishment profile of the SKU, namely the typical number of units received per restocking cycle. Sizing bins to match this quantity minimizes the number of bin accesses required during put-away operations.

\paragraph{SKU Stacking Constraint.}
Items of different SKUs may not be stacked vertically (Z-axis) nor placed one behind the other in depth (Y-axis). This ensures that every SKU stored in a bin is directly accessible from its front opening. Items of different SKUs may only be arranged side by side along the X-axis (bin length). Items of the same SKU, being physically identical, may be organized freely in any configuration. Figure~\ref{fig:valido_invalido} illustrates the three possible relative positions between two blocks of distinct SKUs and identifies the only permitted one.

\begin{figure}[htbp]
    \centering
    \includegraphics[width=\textwidth, trim={0.6cm 0.5cm 0.6cm 1.1cm},
    clip]{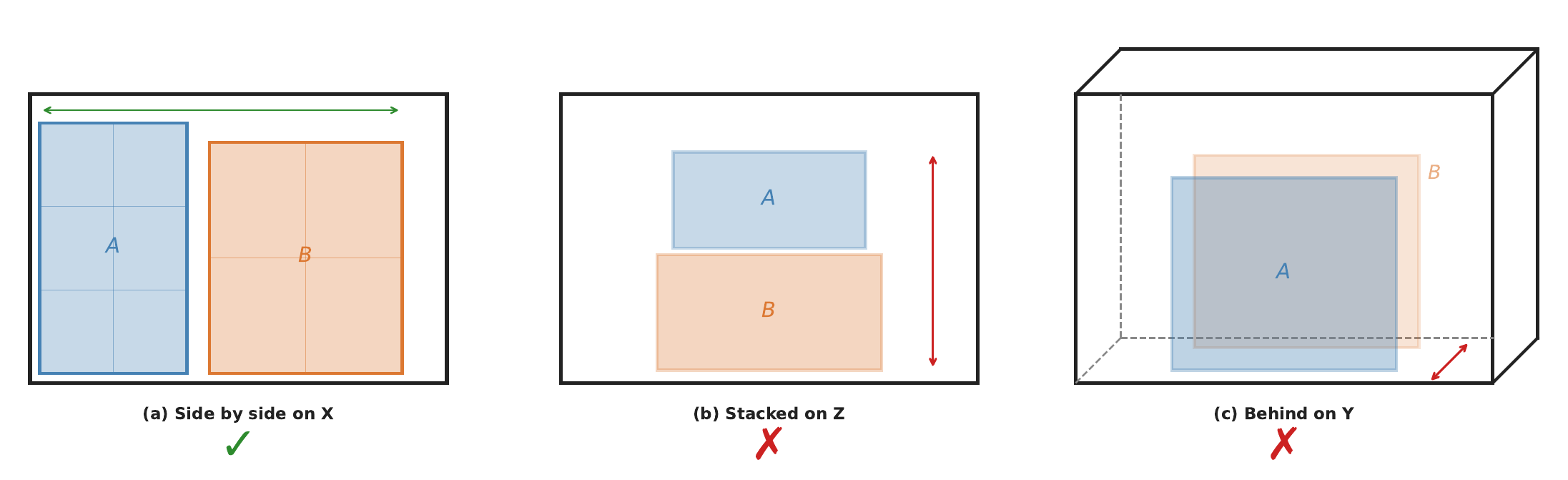}
    \caption{Valid and invalid relative positions for blocks of distinct SKUs. Only side-by-side placement along the X-axis~(a) is permitted; vertical stacking along Z-axis~(b) and depth-wise positioning along Y-axis~(c) are prohibited.}
    \label{fig:valido_invalido}
\end{figure}

As a consequence of these three constraints, every feasible solution consists of bins in which each present SKU occupies a well-defined lateral strip along the X-axis. Figure~\ref{fig:layout_bin} shows a representative example: three blocks of distinct SKUs (A, B, C) are arranged side by side along the X-axis, occupying widths $X^*_{A,t}$, $X^*_{B,t}$, and $X^*_{C,t}$ respectively, whose sum does not exceed $L_t$. The computation of these optimal widths is developed in Section~\ref{sec:decomposition}.

\begin{figure}[htbp]
    \centering
    \includegraphics[width=0.75\textwidth, trim={3.5cm 4.3cm 3.05cm 4.8cm},
    clip]{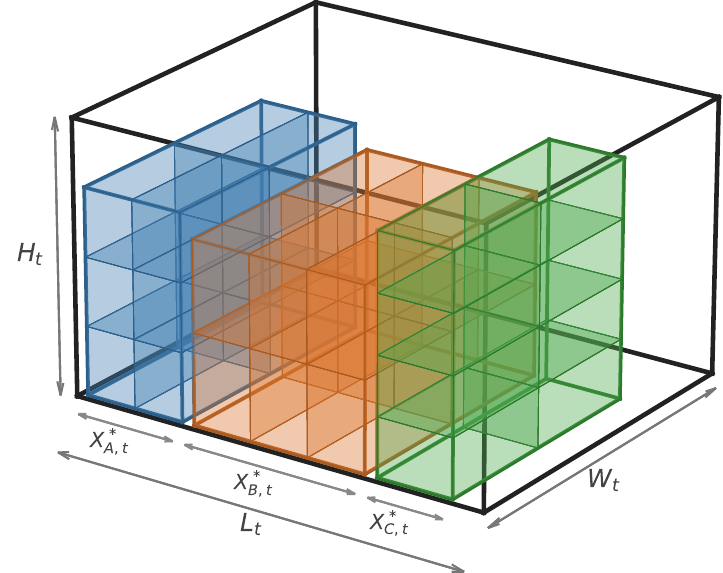}
    \caption{A bin of type $t$ with three blocks of distinct SKUs positioned side by side along the X-axis. The optimal widths $X^*_{A,t}$, $X^*_{B,t}$, $X^*_{C,t}$ satisfy $X^*_{A,t} + X^*_{B,t} + X^*_{C,t} \leq L_t$.}
    \label{fig:layout_bin}
\end{figure}

% remark
Let $V = \sum_{s \in S} d_s\, l_s w_s h_s$ denote the total volume of all items to be stored. This quantity is constant for a given instance, since every item must be stored regardless of the solution. The total bin volume used---the objective being minimized---is precisely the denominator of the average volumetric utilization $\eta = V / \sum_{j} V_{t_j}$, where $V_{t_j}$ is the volume of the bin type assigned to bin $j$. Bins in $J$ that are not used in a solution are assigned zero volume ($V_{t_j}=0$), so the sum ranges over all of $J$ without needing to distinguish used from unused bins at this point. Since $V$ is constant, minimizing the total bin volume is therefore equivalent to maximizing $\eta$.

%==============================================================================
\section{Dimensional Decomposition}
\label{sec:decomposition}
%==============================================================================

This section develops the dimensional decomposition that transforms the original three-dimensional item packing problem into a one-dimensional block positioning problem. The decomposition proceeds in three steps, reducing the problem from positioning individual items in 3D, to positioning 1D blocks.

\subsection{Item Aggregation into Blocks}
\label{subsec:aggregation}

The first reduction follows from observing that items of the same SKU are physically identical and, by the stacking constraint, are the only ones that may be stacked vertically or placed one behind the other in depth. Consequently, all items of a given SKU assigned to the same bin can be treated as a single geometric entity: a block.

\begin{definition}[Block]
A block $(s,k)$ is a group of items of SKU $s$ to be assigned together to a single bin. The index $k \in \{1, \ldots, K_s\}$ distinguishes the different blocks of the same SKU.
\end{definition}

Since each bin may contain at most $n_s$ items of SKU $s$, and there are $d_s$ items of that SKU in total, the minimum number of blocks required is:
\begin{equation}
    K_s = \left\lceil \frac{d_s}{n_s} \right\rceil
    \label{eq:num_bloques}
\end{equation}
When the desired $n_s$ exceeds the maximum number of items of SKU $s$ that fit in any bin of the catalog, $n_s$ is reduced to that maximum, so that every block admits at least one compatible bin type.

\begin{definition}[Block size]
The number of items in block $(s,k)$ is:
\begin{equation}
    q_{sk} = \begin{cases}
        n_s & \text{if } k < K_s \\
        d_s - (K_s - 1) \cdot n_s & \text{if } k = K_s
    \end{cases}
    \label{eq:cantidad_bloque}
\end{equation}
\end{definition}

The first $K_s - 1$ blocks contain exactly $n_s$ items (the maximum allowed), while the last block contains the remaining items.

\begin{definition}[Block set]
The complete set of blocks is:
\begin{equation}
    \mathcal{B} = \{(s, k) : s \in S, \; k \in \{1, \ldots, K_s\}\}
    \label{eq:conjunto_bloques}
\end{equation}
\end{definition}

\begin{proposition}[Reduction bound]
The number of blocks satisfies:
\begin{equation}
    |S| \leq |\mathcal{B}| \leq \sum_{s\in S} d_s
\end{equation}
\end{proposition}
Note that $|S| \leq |\mathcal{B}| \leq \sum_{s\in S} d_s$: the lower bound is attained when every SKU fits in a single block ($d_s\leq n_s$ for all $s$), and the upper bound when $n_s = 1$ for all $s$.

\subsection{Reduction to One-Dimensional Positioning}
\label{subsec:1d_reduction}

The second reduction follows from the combined effect of the stacking constraint and the quantity limit per SKU. Consider two blocks $(s,k)$ and $(s',k')$ of distinct SKUs ($s \neq s'$) assigned to the same bin. By the stacking constraint, only their separation along the X-axis is relevant for non-overlap.

This fact has three important implications for the formulation. First, Y-coordinates are irrelevant for non-overlap, so all blocks can be placed at $y = 0$ (flush with the front of the bin). Second, Z-coordinates are likewise irrelevant, so all blocks rest on the bin floor ($z = 0$). Third, it is sufficient to determine only the X-coordinate of each block and verify that no two blocks overlap along the X-axis.

Although positioning reduces to one dimension, all three block dimensions remain relevant for bin compatibility. Along X, the sum of the X-dimensions of all blocks in a bin must not exceed the bin length $L_t$. Along Y, the Y-dimension of each block must not exceed the bin width $W_t$. Along Z, the height of each block must not exceed the bin height $H_t$.

\subsection{Bounding Box and Block Configurations}
\label{subsec:bounding_box}

Having established that the problem reduces to one-dimensional block positioning, we now characterize the geometric configurations that a block may adopt. A block $(s,k)$ contains $q_{sk}$ identical items of dimensions $(l_s, w_s, h_s)$.

It is important to distinguish between the \emph{occupied region} of the items and the \emph{bounding box} that contains them. Items within a block need not form a solid rectangular prism; they may be arranged with internal gaps, provided they fit within a bounding box that is compatible with the assigned bin. For the purposes of X-axis positioning and bin compatibility, only the bounding box matters.

\begin{definition}[Block bounding box]
The bounding box of a block is the smallest rectangular prism enclosing all its items. A bounding box of dimensions $a \times b \times c$ (in units of items) has capacity for $a \cdot b \cdot c$ positions, of which $q_{sk}$ are occupied by items.
\end{definition}

\begin{definition}[Block configuration]
A configuration $(a, b, c)$ for a block with $q$ items is a triple of positive integers defining a bounding box such that:
\begin{equation}
    a \cdot b \cdot c \geq q
    \label{eq:config_constraint}
\end{equation}
where $a$, $b$, and $c$ denote the number of positions along the X, Y, and Z axes, respectively.
\end{definition}

% remark
The inequality $a \cdot b \cdot c \geq q$ reflects that the bounding box may contain empty positions.

The physical dimensions of the bounding box depend on the \emph{orientation} of the block, i.e., on which item dimension is aligned with each bin axis. An orientation is a triple $(\delta_X, \delta_Y, \delta_Z)$ that assigns one of the item dimensions $(l_s, w_s, h_s)$ to each axis. The set of admissible orientations, denoted by~$\mathcal{O}$, contains the two XY-rotations $\{(l_s, w_s, h_s),\, (w_s, l_s, h_s)\}$ for non-rotatable items, or all six permutations of $(l_s, w_s, h_s)$ for rotatable items. Under orientation $(\delta_X, \delta_Y, \delta_Z)$ and configuration $(a,b,c)$, the physical dimensions of the bounding box are:
\begin{equation}
    \text{X-dimension: } a \cdot \delta_X, \qquad
    \text{Y-dimension: } b \cdot \delta_Y, \qquad
    \text{Z-dimension: } c \cdot \delta_Z
    \label{eq:dims_config}
\end{equation}

Figure~\ref{fig:bounding_box} illustrates an example with $q = 7$ items in a $2 \times 2 \times 2 = 8$-position bounding box, leaving one position empty.

\begin{figure}[htbp]
    \centering
    \includegraphics[width=0.65\textwidth, trim={5.5cm 4.7cm 3.7cm 6.3cm}, clip]{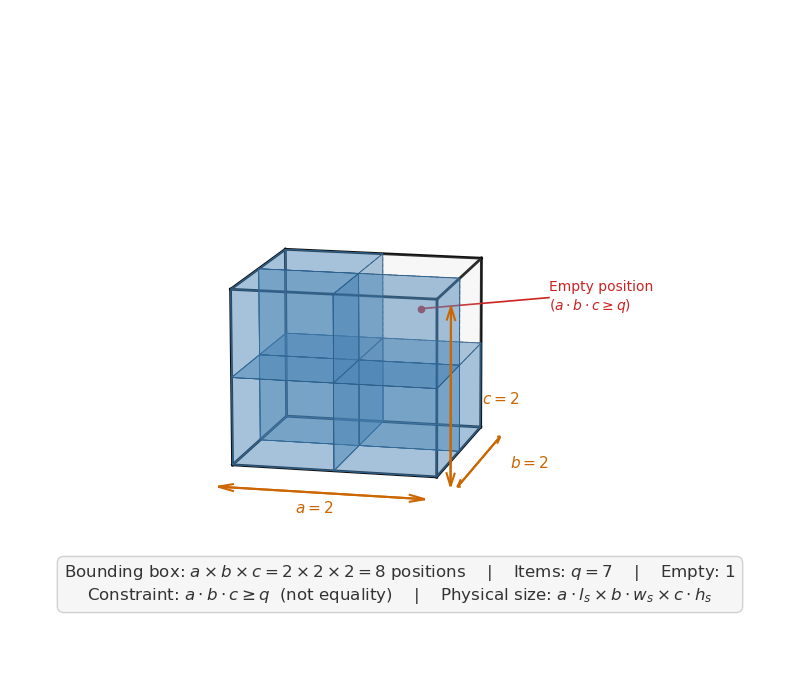}
    \caption{Bounding box of dimensions $a \times b \times c = 2 \times 2 \times 2$ for a block with $q = 7$ items. The empty position reflects the inequality $a \cdot b \cdot c \geq q$. The physical size of the bounding box is $a \cdot l_s \times b \cdot w_s \times c \cdot h_s$.}
    \label{fig:bounding_box}
\end{figure}

\begin{definition}[Extended configuration]
An extended configuration $(a, b, c, \delta_X, \delta_Y, \delta_Z)$ specifies both the bounding box $(a,b,c)$ and the orientation $(\delta_X, \delta_Y, \delta_Z) \in \mathcal{O}$.
\end{definition}

\begin{definition}[Set of valid configurations]
For a block $(s,k)$ with $q_{sk}$ items and a bin type $t$, the set of valid configurations is:
\begin{equation}
    \mathcal{C}_{sk,t} = \left\{(a,b,c,\delta_X,\delta_Y,\delta_Z) \in \mathbb{Z}^3_+
    \times \mathcal{O} :
    \begin{array}{l}
        a \cdot b \cdot c \geq q_{sk} \\[0.3em]
        a \cdot \delta_X \leq L_t \\[0.3em]
        b \cdot \delta_Y \leq W_t \\[0.3em]
        c \cdot \delta_Z \leq H_t
    \end{array}
    \right\}
    \label{eq:configs_validas}
\end{equation}
\end{definition}

\subsection{Optimal Configuration}
\label{subsec:optimal_config}

A key simplification arises from the following observation: given the structure of the problem, the optimal configuration of each block is uniquely determined by the bin type to which it is assigned.

By the preceding reductions, packing amounts to positioning blocks along the X-axis only. The Y- and Z-dimensions of each block need only satisfy bin compatibility constraints (fitting within the bin width and height); they play no role in the non-overlap constraints between blocks.

\begin{proposition}[Optimality of minimizing the X-dimension]
\label{prop:min_x}
Let $(s,k)$ be a block assigned to a bin of type $t$. Among all valid configurations in $\mathcal{C}_{sk,t}$, the optimal one is that which minimizes the X-dimension of the block.
\end{proposition}

The following proposition establishes that this minimum X-dimension can be computed efficiently without enumerating all candidate configurations.

\begin{proposition}[Structure of the optimal configuration]
\label{prop:estructura_optima}
For a block $(s,k)$ with $q_{sk}$ items of SKU $s$ (dimensions $l_s, w_s, h_s$) and a bin of type $t$ (dimensions $L_t, W_t, H_t$), the minimum X-dimension is achieved by maximizing the bounding box capacity in the Y and Z directions.
\end{proposition}

\begin{proof}
Fix an orientation with effective item dimensions $(\delta_X, \delta_Y, \delta_Z)$, where $(\delta_X, \delta_Y, \delta_Z)$ is one of the admissible permutations of $(l_s, w_s, h_s)$. Let $b_{\max} = \lfloor W_t / \delta_Y \rfloor$ and $c_{\max} = \lfloor H_t / \delta_Z \rfloor$ be the maximum number of items that fit along the Y- and Z-axes respectively. Each slice perpendicular to the X-axis has capacity $b \cdot c \leq b_{\max} \cdot c_{\max}$ positions, so the minimum number of slices required to contain $q_{sk}$ items is $a_{\min} = \lceil q_{sk} / (b_{\max} \cdot c_{\max}) \rceil$, which is attained by setting $b = b_{\max}$ and $c = c_{\max}$. The X-dimension under this orientation is $a_{\min} \cdot \delta_X$. The optimal X-dimension is then the minimum over all admissible orientations.
\end{proof}

This result enables the optimal X-dimension to be computed in $O(1)$ time. Concretely, the optimal X-dimension $X^*_{sk,t}$ is obtained by evaluating each admissible orientation: for each, the maximum slice capacity $b_{\max} \cdot c_{\max}$ is computed from $b_{\max} = \lfloor W_t / \delta_Y \rfloor$ and $c_{\max} = \lfloor H_t / \delta_Z \rfloor$ (where $(\delta_X, \delta_Y, \delta_Z)$ are the effective item dimensions under the given orientation), and the minimum number of X-slices is $a_{\min} = \lceil q_{sk} / (b_{\max} \cdot c_{\max}) \rceil$. The optimal X-dimension is then:
\begin{equation}
    X^*_{sk,t} = \min_{(\delta_X, \delta_Y, \delta_Z) \in \mathcal{O}} \;
    a_{\min} \cdot \delta_X,
    \label{eq:X_opt}
\end{equation}
where $\mathcal{O}$ is the set of admissible orientations defined in Section~\ref{subsec:bounding_box}. The argument of Propositions~\ref{prop:min_x} and~\ref{prop:estructura_optima} applies independently to each orientation in $\mathcal{O}$; the minimum over all orientations yields the globally optimal X-dimension. We set $X^*_{sk,t} = +\infty$ when no feasible orientation exists, i.e., block $(s,k)$ is incompatible with type~$t$. We denote by $T_{sk} = \{t \in T : 
X^*_{sk,t} < +\infty\}$ the set of bin types compatible with block $(s,k)$. Algorithm~\ref{alg:bbox} formalizes the complete procedure.

\begin{algorithm}[H]
\caption{Optimal Bounding Box Computation}
\label{alg:bbox}
\begin{algorithmic}[1]
\REQUIRE Block with $q$ items of dimensions $(l_s, w_s, h_s)$, bin of type $t$ with dimensions $(L_t, W_t, H_t)$, rotatability flag
\ENSURE Optimal X-dimension $X^*$, or $\infty$ if infeasible
%\STATE $X^* \leftarrow \infty$
%\FOR{each orientation $(\delta_X, \delta_Y, \delta_Z) \in \text{orientations}(l_s, w_s, h_s, \text{rotatable})$}
\STATE $X^* \leftarrow \infty$
\STATE $\mathcal{O} \leftarrow$ six permutations of $(l_s,w_s,h_s)$ if rotatable, else $\{(l_s,w_s,h_s),\,(w_s,l_s,h_s)\}$
\FOR{each orientation $(\delta_X, \delta_Y, \delta_Z) \in \mathcal{O}$}
    \STATE $a_{\max} \leftarrow \lfloor L_t / \delta_X \rfloor$, \quad $b_{\max} \leftarrow \lfloor W_t / \delta_Y \rfloor$, \quad $c_{\max} \leftarrow \lfloor H_t / \delta_Z \rfloor$
    \IF{$a_{\max} \geq 1$ \AND $b_{\max} \geq 1$ \AND $c_{\max} \geq 1$}
        \STATE $a_{\min} \leftarrow \lceil q / (b_{\max} \cdot c_{\max}) \rceil$
        \IF{$a_{\min} \leq a_{\max}$}
            \STATE $X^* \leftarrow \min(X^*, \; a_{\min} \cdot \delta_X)$
        \ENDIF
    \ENDIF
\ENDFOR
\RETURN $X^*$
\end{algorithmic}
\end{algorithm}

The orientation set contains all six permutations of $(l_s,w_s,h_s)$ for fully rotatable items. For non-rotatable items, the height dimension remains aligned with the Z-axis and only the two horizontal rotations $(l_s,w_s,h_s)$ and $(w_s,l_s,h_s)$ are allowed. A return value of $X^* = \infty$ indicates that no orientation allows the $q$ items to fit within the bin. The computation can be accelerated by deduplication: blocks sharing the same tuple $(q,\, l_s,\, w_s,\, h_s,\, \text{rotatable})$ yield identical results for every bin type. Computing $X^*_{sk,t}$ requires $O(1)$ arithmetic operations per (block, type) pair; for the full problem, preprocessing all parameters $X^*_{sk,t}$ requires $O(|\mathcal{B}| \cdot |T|)$ operations in total.

%==============================================================================
\section{MILP Formulation}
\label{sec:compact_formulation}
%==============================================================================

This section presents a compact MILP formulation for the bin dimensioning problem, built on the one-dimensional block representation developed in Section~\ref{sec:decomposition}. The  sets, parameters, and decision variables of the model are described next.

Table~\ref{tab:conjuntos} presents the fundamental sets used in the formulation. The set $S$ represents the SKUs of the problem, $\mathcal{B}$ denotes the set of blocks, $J$ indexes the bins sequentially, and $T$ represents the available bin types.

\begin{table}[H]
\centering
\caption{Model sets.}
\label{tab:conjuntos}
\begin{tabular}{cl}
\toprule
\textbf{Set} & \textbf{Description} \\
\midrule
$S$ & SKUs \\
$\mathcal{B}$ & Blocks: $\mathcal{B} = \{(s, k) : s \in S,\; k \in \{1, \ldots, K_s\}\}$ \\
$J$ & Bins (indexed sequentially) \\
$T$ & Available bin types \\
\bottomrule
\end{tabular}
\end{table}

Table~\ref{tab:parametros} details the remaining model parameters. For each SKU $s$, the item dimensions $(l_s, w_s, h_s)$ are known. Each block $(s,k)$ contains $q_{sk}$ items, and bin types have dimensions $(L_t, W_t, H_t)$. The parameter $m$ specifies the maximum number of distinct SKUs allowed per bin, and $K_s = \lceil d_s / n_s \rceil$ denotes the number of blocks of SKU $s$. Finally, $X^*_{sk,t}$ is the optimal X-dimension of block $(s,k)$ when assigned to a bin type $t$, computed as described in Section~\ref{subsec:optimal_config}.

\begin{table}[H]
\centering
\caption{Model parameters.}
\label{tab:parametros}
\begin{tabular}{cl}
\toprule
\textbf{Parameter} & \textbf{Description} \\
\midrule
$l_s, w_s, h_s$ & Dimensions of the items of SKU $s$ \\
$q_{sk}$ & Number of items in block $(s,k)$ \\
$L_t, W_t, H_t$ & Dimensions of bin type $t$ \\
$m$ & Maximum number of distinct SKUs per bin \\
$K_s$ & Number of blocks of SKU $s$: $K_s = \lceil d_s / n_s \rceil$ \\
$X^*_{sk,t}$ & Optimal X-dimension of block $(s,k)$ assigned to type $t$ \\
\bottomrule
\end{tabular}
\end{table}

Table~\ref{tab:variables} presents the decision variables. The binary variable $p_{skj}$ indicates whether block $(s,k)$ is assigned to bin $j$; $u_j$ indicates whether bin $j$ is in use; $b_{jt}$ indicates whether bin $j$ is of type $t$; and $\omega_{sj}$ indicates whether SKU $s$ is present in bin $j$.

\begin{table}[H]
\centering
\caption{Decision variables.}
\label{tab:variables}
\begin{tabular}{ccl}
\toprule
\textbf{Variable} & \textbf{Domain} & \textbf{Description} \\
\midrule
$p_{skj}$ & $\{0,1\}$ & 1 if block $(s,k)$ is assigned to bin $j$ \\
$u_j$     & $\{0,1\}$ & 1 if bin $j$ is in use \\
$b_{jt}$  & $\{0,1\}$ & 1 if bin $j$ is of type $t$ \\
$\omega_{sj}$ & $\{0,1\}$ & 1 if SKU $s$ is present in bin $j$ \\
\bottomrule
\end{tabular}
\end{table}

The formulation is as follows:
\begin{align}
    \min \quad & \sum_{j \in J} \sum_{t \in T} b_{jt} \cdot L_t W_t H_t
    \label{obj} \\[1em]
    % --
    \text{s.t.} \quad
    & \sum_{j \in J} p_{skj} = 1,
      && \forall (s,k) \in \mathcal{B},
      \label{c:asig_bloque} \\[0.5em]
        % --
    & p_{skj} \leq u_j,
      && \forall (s,k) \in \mathcal{B}, \; j \in J,
      \label{c:activ_bin} \\[0.5em]
          % --
    & \sum_{t \in T} b_{jt} = u_j,
      && \forall j \in J,
      \label{c:tipo_bin} \\[0.5em]
          % --
    &\omega_{sj} \geq \sum_{k=1}^{K_s} p_{skj},
    &&\forall s \in S, \; j \in J,
    \label{c:presencia_sku} \\[0.5em]
          % --
    & \sum_{s \in S} \omega_{sj} \leq m,
      && \forall j \in J,
      \label{c:max_skus} \\[0.5em]
          % --
    & \sum_{k=1}^{K_s} p_{skj} \leq 1,
      && \forall s \in S, \; j \in J,
      \label{c:unicidad_bloque} \\[0.5em]
          % --
    & \sum_{(s,k) \in \mathcal{B}} \sum_{t \in T}
        X^*_{sk,t} \cdot p_{skj} \cdot b_{jt}
      \leq \sum_{t \in T} L_t \cdot b_{jt},
      && \forall j \in J,
      \label{c:capacidad} \\[0.5em]
          % --
    & u_j \geq u_{j+1},
      && \forall j \in J \setminus \{|J|\},
      \label{c:simetria} \\[0.5em]
          % --
    & p_{skj},\, u_j,\, b_{jt},\, \omega_{sj} \in \{0,1\}.
      \label{c:binarias}
\end{align}
The objective function~\eqref{obj} minimizes the total volume of bins used. Since
$b_{jt} = 1$ only if bin $j$ is of type $t$, and each active bin is assigned
exactly one type by constraint~\eqref{c:tipo_bin}, the sum
$\sum_{t \in T} b_{jt} \cdot L_t W_t H_t$ equals the volume of bin $j$ when
it is in use, and zero otherwise. Constraints~\eqref{c:asig_bloque} are partition constraints ensuring that every
block is assigned to exactly one bin. Constraints~\eqref{c:activ_bin} link
block assignments to bin activation: if any block is placed in bin $j$, then
$u_j$ must equal one. Constraints~\eqref{c:tipo_bin} complete the bin
activation logic by requiring that each active bin be assigned exactly one type
($\sum_t b_{jt} = 1$ when $u_j = 1$) while leaving no type selected for
inactive bins ($\sum_t b_{jt} = 0$ when $u_j = 0$).

Constraints~\eqref{c:presencia_sku} define $\omega_{sj}$ as a presence
indicator: the variable is forced to one whenever any block of SKU $s$ is
assigned to bin $j$. Constraints~\eqref{c:max_skus} then use these indicators
to enforce the SKU diversity limit, requiring that the number of distinct SKUs
present in any bin not exceed $m$. Constraints~\eqref{c:unicidad_bloque}
guarantee that each bin holds at most one block per SKU, which follows directly
from the block construction: since each block already contains up to $n_s$
items---the maximum allowed per bin---two blocks of the same SKU cannot coexist
in the same bin.

Constraints~\eqref{c:capacidad} enforce X-axis capacity: the sum of the optimal
X-dimensions of all blocks assigned to bin $j$ must not exceed its length. The
product $p_{skj} \cdot b_{jt}$, which equals one if and only if block $(s,k)$
is assigned to bin $j$ and that bin is of type $t$, is bilinear in binary
variables and is linearized via standard McCormick inequalities. The linearization is complemented by the compatibility constraint
\begin{align}
    p_{skj} \leq \sum_{t \in T_{sk}} b_{jt},
    \qquad \forall (s,k) \in \mathcal{B},\; j \in J,
    \label{eq:compatibilidad}
\end{align}
which restricts each block to bins of geometrically compatible
types. Without it, a block assigned to an incompatible type would not be
detected as infeasible by~\eqref{c:capacidad}, since the McCormick variables
are only instantiated for compatible pairs.

Constraints~\eqref{c:simetria} break the bin-permutation symmetry by enforcing
sequential activation: if bin $j+1$ is in use, then bin $j$ must also be in
use, so any solution using $k$ bins uses exactly $\{1, \ldots, k\}$. Finally, the number of bins $|J|$ is set to an upper bound on the number of
bins required by any feasible solution; in practice, any feasible solution
provides a valid value for $|J|$.

The formulation's size is dominated by the $|\mathcal{B}| \cdot |J|$ binary
variables $p_{skj}$; the bin-activation constraints~\eqref{c:activ_bin} and
SKU-presence constraints~\eqref{c:presencia_sku} each contribute a further
$|\mathcal{B}| \cdot |J|$ inequalities, giving an overall
$O(|\mathcal{B}| \cdot |J|)$ scaling. Since both $|\mathcal{B}|$ and $|J|$ are
of order $10^5$--$10^6$ for real instances, direct solution is intractable,
motivating the  column generation and heuristic approaches developed in
Sections~\ref{sec:column_generation} and~\ref{sec:bfd}, respectively.

%==============================================================================
\section{Column Generation}
\label{sec:column_generation}
%==============================================================================

The compact formulation~\eqref{obj}--\eqref{c:binarias} has $O(|\mathcal{B}| \cdot |J|)$ variables and constraints, rendering direct solution intractable for real-scale instances. In this section, we develop a column generation scheme that exploits a set partitioning reformulation to produce tight LP lower bounds against which any heuristic solution can be assessed.

The reformulation is based on the concept of a \textit{packing pattern}: a complete bin configuration specifying a type $t_p \in T$ and a set of blocks $\mathcal{B}_p \subseteq \mathcal{B}$, with cost $c_p = L_{t_p} W_{t_p} H_{t_p}$. A pattern is \emph{valid} if it satisfies the SKU limit ($|\{s : (s,k) \in \mathcal{B}_p\}| \leq m$), SKU uniqueness (at most one block per SKU), and X-capacity ($\sum_{(s,k) \in \mathcal{B}_p} X^*_{sk,t_p} \leq L_{t_p}$).

Let $\mathcal{P}$ denote the (exponentially large) set of all valid patterns and let $y_p \in \{0,1\}$ indicate whether pattern~$p$ is selected. The bin dimensioning problem is equivalently stated as:
\begin{align}
    \min \quad & \sum_{p \in \mathcal{P}} c_p \, y_p
    \label{eq:sp_obj} \\
    \text{s.t.} \quad
    & \sum_{p \in \mathcal{P}:\, (s,k) \in \mathcal{B}_p} y_p = 1,
      \quad \forall (s,k) \in \mathcal{B}, \label{eq:sp_cov} \\
    & y_p \in \{0,1\}, \quad \forall p \in \mathcal{P}. \label{eq:sp_bin}
\end{align}
Every feasible solution of the compact model corresponds to a selection of patterns in~\eqref{eq:sp_obj}--\eqref{eq:sp_bin}, and vice versa.

Since $\mathcal{P}$ cannot be enumerated explicitly, column generation operates on a \emph{Restricted Master Problem} (RMP): the LP relaxation of~\eqref{eq:sp_obj}--\eqref{eq:sp_bin} restricted to a subset $\mathcal{P}' \subseteq \mathcal{P}$. The RMP is initialized with a set of feasible patterns. At each iteration the RMP is solved, yielding dual values $\pi^*_{sk}$ for the covering constraints~\eqref{eq:sp_cov}. The \emph{pricing subproblem} (SP) then searches for patterns with negative reduced cost $\bar{c}_p = c_p - \sum_{(s,k) \in \mathcal{B}_p} \pi^*_{sk} < 0$. If none exists, the current RMP solution is optimal for the full LP relaxation; otherwise, the best candidates are added to $\mathcal{P}'$ and the process repeats.

The SP seeks the pattern that minimizes the reduced cost, or equivalently, that maximizes $\sum_{(s,k) \in \mathcal{B}_p} \pi^*_{sk} - c_p$. Given dual values $\pi^*_{sk}$ from the RMP, the SP solves:
\begin{align}
    \max \quad
    & \sum_{(s,k) \in \mathcal{B}} \pi^*_{sk} \cdot p_{sk}
      - \sum_{t \in T} V_t \cdot b_t
    \label{eq:sp_obj_full} \\[0.5em]
    \text{s.t.} \quad
    & \sum_{t \in T} b_t = 1, \label{eq:sp_tipo} \\[0.5em]
    & \omega_s \geq p_{sk},
      && \forall s \in S,\; k \in \{1,\ldots,K_s\},
      \label{eq:sp_presencia} \\[0.5em]
    & \sum_{s \in S} \omega_s \leq m, \label{eq:sp_max_skus} \\[0.5em]
    & \sum_{k=1}^{K_s} p_{sk} \leq 1,
      && \forall s \in S, \label{eq:sp_unicidad} \\[0.5em]
    & \sum_{(s,k) \in \mathcal{B}} \sum_{t \in T}
        X^*_{sk,t} \cdot p_{sk} \cdot b_t
      \leq \sum_{t \in T} L_t \cdot b_t,
      \label{eq:sp_capacidad} \\[0.5em]
    & p_{sk},\, b_t,\, \omega_s \in \{0,1\}.
      \label{eq:sp_binary}
\end{align}
Rather than solving~\eqref{eq:sp_obj_full}--\eqref{eq:sp_binary} directly, we adopt a \emph{per-type pricing} strategy: for each $t \in T$ we fix $b_t = 1$ and $b_{t'} = 0$ for all $t' \neq t$, and optimize only over the block selection. Fixing the type eliminates the $b_t$ variables and linearizes the capacity constraint~\eqref{eq:sp_capacidad}, since $X^*_{sk,t}$ is a constant, yielding $|T|$ independent subproblems each substantially more compact than~\eqref{eq:sp_obj_full}--\eqref{eq:sp_binary}. When the optimal value of the subproblem for type~$t$ exceeds zero, the selected blocks define a pattern with negative reduced cost.

Before constructing the MILP for a given type~$t$, two necessary conditions are checked to avoid solving subproblems that cannot produce improving patterns. First, if no compatible block has a positive dual value ($\max_{(s,k) \in \mathcal{B}_t} \pi^*_{sk} \leq 0$), no subset of blocks can yield a positive dual benefit, and type~$t$ is skipped. Second, let $\hat{\pi}_s = \max_{k:\,(s,k) \in \mathcal{B}_t} \pi^*_{sk}$ be the best dual value among the blocks of SKU~$s$ compatible with~$t$, and let $\hat{\pi}^{(1)} \geq \hat{\pi}^{(2)} \geq \cdots \geq \hat{\pi}^{(|S_t|)}$ be these values in decreasing order. If $\sum_{i=1}^{\min(m,\,|S_t|)} \hat{\pi}^{(i)} \leq V_t$, then even the most favorable selection of $m$~SKUs cannot produce a dual benefit exceeding the bin cost, and type~$t$ is likewise skipped. These checks become increasingly effective as the duals stabilize in later iterations.

To accelerate convergence, each per-type SP is solved with Gurobi's solution
pool, extracting up to $k$ candidate patterns per type. The candidates
collected across all types are pooled and sorted by reduced cost. The parameter
$P_{\max}$ bounds the number of columns added to the RMP per iteration: the
$P_{\max}$ candidates with the most negative reduced cost are retained, and
among these, any pattern already present in $\mathcal{P}'$ is filtered out by
fingerprint-based duplicate detection. The number of columns effectively added
in an iteration may therefore be smaller than $P_{\max}$. The algorithm
terminates when no candidate with reduced cost below $-\epsilon$ is
found---equivalently, when no new column is added---or when the maximum
iteration count $I_{\max}$ is reached.

\begin{algorithm}[H]
\caption{Column Generation for Bin Dimensioning}
\label{alg:column_generation}
\begin{algorithmic}[1]
\REQUIRE Block set $\mathcal{B}$, bin types $T$, initial patterns $\mathcal{P}'$, tolerance $\epsilon$, limits $I_{\max}$, $P_{\max}$, pool size $k$
\ENSURE LP lower bound $z^*_{\text{LP}}$, pattern set $\mathcal{P}'$
\STATE Build RMP from $\mathcal{P}'$
\FOR{$\text{iter} = 1, \ldots, I_{\max}$}
    \STATE Solve RMP $\rightarrow$ objective $z_{\text{RMP}}$, duals $\pi^*$
    \STATE $\text{Candidates} \leftarrow \emptyset$
    \FOR{each type $t \in T$ not eliminated by pre-filtering}
        \STATE Solve per-type SP with $b_t = 1$ fixed, pool of size $k$
        \STATE $\text{Candidates} \leftarrow \text{Candidates} \cup \{p : \bar{c}_p < -\epsilon\}$
    \ENDFOR
    \STATE Add the best $\min(|\text{Candidates}|, P_{\max})$ non-duplicate patterns to $\mathcal{P}'$
    \IF{no pattern added}
        \STATE \textbf{break} \COMMENT{LP optimality reached}
    \ENDIF
\ENDFOR
\RETURN $z^*_{\text{LP}} \leftarrow z_{\text{RMP}}$, $\mathcal{P}'$
\end{algorithmic}
\end{algorithm}

The LP lower bound satisfies $z^*_{\text{LP}} \leq z^*_{\text{OPT}}$, where $z^*_{\text{OPT}}$ is the optimal integer value. Once column generation converges, the pattern set $\mathcal{P}'$ may be used to solve a \emph{Restricted Integer Program} (RIP): the integer version of~\eqref{eq:sp_obj}--\eqref{eq:sp_bin} restricted to $\mathcal{P}'$, with covering constraints ($\geq 1$) replacing the partitioning constraints ($= 1$) to guarantee feasibility without branch-and-price. 
The solver is warm-started using the nearly integer LP solution, setting
$y_p = 1$ in the MIP start for all patterns with $y^*_p \geq 0.999$.

%==============================================================================
\section{Best-Fit-Decreasing Heuristic}
\label{sec:bfd}
%==============================================================================

The column generation framework of Section~\ref{sec:column_generation} provides tight lower bounds but requires substantial computation time. For large-scale instances, a fast heuristic is needed both as a standalone method and as an initializer for column generation. In this section we present a three-phase algorithm based on the Best-Fit-Decreasing (BFD) strategy~\citep{johnson1973near,coffman1996approximation}, adapted to the bin dimensioning problem.

\paragraph{Isolated block pre-filter.}
A block $(s,k)$ is \emph{isolated} if, for every compatible bin type $t \in T_{sk}$, the residual X-space after placing the block cannot accommodate the smallest block of any other SKU. Such blocks will necessarily occupy a bin alone in any feasible solution. To test this efficiently, the algorithm precomputes, for each bin type~$t$, the two smallest $X^*$ values across compatible SKUs---call them $\hat{X}^{(1)}_t$ and $\hat{X}^{(2)}_t$, corresponding to distinct SKUs. A block $(s,k)$ is isolated if $\hat{X}^{\mathrm{other}}_{s,t} > L_t - X^*_{sk,t}$ for all $t \in T_{sk}$, where $\hat{X}^{\mathrm{other}}_{s,t}$ equals $\hat{X}^{(1)}_t$ when the SKU achieving that minimum is not~$s$, and $\hat{X}^{(2)}_t$ otherwise. This precomputation runs in $O(|\mathcal{F}|)$ time, where $|\mathcal{F}|$ is the number of feasible block--type pairs. Isolated blocks are assigned directly to their smallest compatible type (by volume) and removed from the BFD phase.

\paragraph{Best-Fit-Decreasing with round-robin ordering.}
The remaining blocks are processed by Best-Fit-Decreasing. The algorithm maintains a set of open bins, each characterized by its type, the blocks assigned so far, and the residual X-length. For each incoming block, it selects the compatible open bin that minimizes the residual length after insertion; if no open bin can accommodate the block, a new bin of the smallest compatible type is opened. An open bin is closed when its residual space falls below a threshold, when it reaches the SKU limit~$m$, or when no remaining block can fit in it.

In classical BFD, items are processed in decreasing order of size~\citep{johnson1973near}. Here, however, the SKU uniqueness constraint (at most one block per SKU per bin) would cause a simple decreasing-volume ordering to process all blocks of the same SKU consecutively, each opening a new bin and leaving those bins underutilized until blocks of other SKUs are processed. To mitigate this, blocks are processed in an \emph{interleaved round-robin} order: SKUs are sorted by the volume of their largest block, and processing alternates among SKUs in successive rounds, taking one block per SKU per round in decreasing internal order. This interleaving ensures that consecutive blocks come from distinct SKUs, promoting immediate bin reuse.

To speed up the search for compatible open bins, an index structure organizes bins by type, restricting the search to the types compatible with the current block. Algorithm~\ref{alg:indexed_search} details the procedure. The index maintains three dictionaries: $\texttt{by\_type}[t]$, the list of open bins of type~$t$; $\texttt{bins\_with\_sku}[s]$, the set of bin identifiers containing at least one block of SKU~$s$; and $\texttt{by\_id}[j]$, a direct reference to open bin~$j$ for $O(1)$ access. The per-block search cost is $O(|T_{sk}| \cdot \bar{k})$, where $\bar{k}$ is the average number of open bins per type.

\begin{algorithm}[H]
\caption{Indexed Search for Compatible Bins}
\label{alg:indexed_search}
\begin{algorithmic}[1]
\REQUIRE Block $(s,k)$, open bin index, parameter $m$
\ENSURE List of compatible open bins
\STATE $\text{Compatible} \leftarrow []$
\STATE $\text{Excluded} \leftarrow \texttt{bins\_with\_sku}[s]$
\FOR{each type $t \in T_{sk}$}
    \STATE $x^{*} \leftarrow X^{*}_{sk,t}$
    \FOR{each bin $j \in \texttt{by\_type}[t]$}
        \IF{$j \in \text{Excluded}$}
            \STATE \textbf{continue} \COMMENT{SKU uniqueness}
        \ENDIF
        \IF{$|\mathcal{S}_j| \geq m$}
            \STATE \textbf{continue} \COMMENT{SKU limit}
        \ENDIF
        \IF{$x^* > L_j^{\text{res}}$}
            \STATE \textbf{continue} \COMMENT{Insufficient space}
        \ENDIF
        \STATE $\text{Compatible} \leftarrow \text{Compatible} \cup \{j\}$
    \ENDFOR
\ENDFOR
\RETURN Compatible
\end{algorithmic}
\end{algorithm}

\paragraph{Solution assembly.}
The bins from the pre-filter and BFD phases are merged and renumbered to produce the final solution. The global volumetric efficiency is $\eta = V \,/\, \sum_j V_{t_j}$, with $V$ defined in Section~\ref{sec:problem_definition}.

%==============================================================================
\section{Computational Experiments}
\label{sec:experiments}
%==============================================================================

This section evaluates the methodology on four synthetic datasets generated from proprietary real-world data provided by the e-commerce company. We first describe the datasets and experimental setup, then present scalability experiments on controlled sub-instances of increasing size, and finally report full-scale results for the complete datasets.

%-----------------------------------------------------------------------
\subsection{Datasets and Experimental Setup}
\label{subsec:datasets_setup}
%-----------------------------------------------------------------------

The experiments employ four synthetic datasets generated from proprietary real-world data collected from the company's distribution centers, spanning a range of scales and product profiles.

The four datasets are designed to represent qualitatively distinct
distribution center profiles observed in the company's network,
differing in scale, product category, demand concentration, and
replenishment patterns. This diversity is
intentional: rather than testing the methodology on homogeneous
instances, the experiments are designed to stress-test it across
operational profiles that differ structurally, so that the results
reflect the range of conditions the approach would face in
deployment.

Each dataset specifies, for every SKU~$s$, the tuple $(l_s,\, w_s,\, h_s,\, d_s,\, n_s,\, \text{rotatable})$ together with the corresponding bin catalog~$T$. 

The synthetic datasets used in the experiments are publicly available at \url{https://github.com/anticiclon/bin-packing-1d}.

The generation process was designed to preserve the main statistical characteristics relevant to the bin dimensioning problem, including the distributions of item dimensions, SKU demand levels, per-bin quantity limits, rotatability rates, and overall instance scales, while avoiding disclosure of sensitive operational information.

The catalogs of available bin types were defined in collaboration with the
engineering team from the company, reflecting structural constraints of the
shelving systems, equipment supplier specifications, and compatibility with
existing distribution center infrastructure. The resulting catalogs contain
24~types for small products and 30~types for large products.

The catalog of small products $\mathcal{T}_{\text{small}}$ is the Cartesian
product of lengths $L_t \in \{38, 46, 57\}$~cm, heights $H_t \in \{25, 35, 45, 55\}$~cm, and widths $W_t \in \{30, 60\}$~cm, yielding $2 \times 4 \times 3 = 24$ types with volumes ranging from 28{,}500~cm$^3$ to 188{,}100~cm$^3$. 
The catalog of large products $\mathcal{T}_{\text{large}}$ uses $L_t \in \{57, 76, 115\}$~cm, $H_t \in \{50, 70, 100, 150, 200\}$~cm, and $W_t \in \{70, 100\}$~cm, giving $2 \times 5 \times 3 = 30$ types with volumes from
199{,}500~cm$^3$ to 2{,}300{,}000~cm$^3$.

The datasets are classified into two size categories: \emph{large products} (DS-1 and DS-2), which use the catalog $\mathcal{T}_{\text{large}}$, and \emph{small products} (DS-3 and DS-4), which use the catalog $\mathcal{T}_{\text{small}}$. 
Table~\ref{tab:datasets_skus} reports the dataset characteristics: the dataset
name, the category of the products in the catalog, the number of SKUs, the
number of items, the percentage of SKUs with unit demand ($d_s=1$), the
percentage of rotatable products and the number of bin types.
Together, these profiles span four orders of magnitude in instance size and a wide range of demand concentration: the percentage of SKUs with unit demand grows from 5.27\% in DS-1 to 23.30\% in DS-2, 39.51\% in DS-3, and 57.80\% in DS-4, confirming that DS-4 is the dataset structurally closest to a classical variable-sized bin packing setting.

\begin{table}[htbp]
\centering
\caption{Dataset characteristics.}
\label{tab:datasets_skus}
\begin{tabular}{@{}llrrrrr@{}}
\toprule
\textbf{Dataset} & \textbf{Catalog} & \textbf{$|S|$} & \textbf{Items} & \textbf{\% SKUs $d_s=1$} & \textbf{Rotatable} & \textbf{$|T|$} \\
\midrule
DS-1 & large & 1{,}556 & 308{,}059 & 5.27\% & 97.0\% & 30 \\
DS-2 & large & 94{,}556 & 905{,}637 & 23.30\% & 99.8\% & 30 \\
DS-3 & small & 503{,}931 & 2{,}218{,}609 & 39.51\% & 96.0\% & 24 \\
DS-4 & small & 1{,}520{,}442 & 4{,}034{,}196 & 57.80\% & 100.0\% & 24 \\
\bottomrule
\end{tabular}
\end{table}

Table~\ref{tab:datasets_bloques} summarizes the block aggregation results. DS-1 has 1,556 SKUs but generates 11,830 blocks (7.6 per SKU on average), the highest ratio among the datasets. The reason behind it is that several of its SKUs have very large item counts (up to 11,499) that exceed the per-bin limit $n_s$ and must be split into many blocks. DS-2 has an intermediate ratio, which is 2.1, while DS-3 and DS-4 exhibit a nearly 1:1 ratio, since most of their SKUs fit within a single block.

Configuration deduplication substantially reduces the number of configurations
that must be evaluated in all datasets, although the magnitude of the reduction
varies with the repetition of item dimensions, block sizes, and rotatability
patterns. DS-1 reduces 11,830 blocks to 2,145 unique configurations, DS-2
reduces 198,640 blocks to 114,869 unique configurations, DS-3 reduces 553,822
blocks to 352,162 unique configurations, and DS-4 reduces 1{,}725{,}931 blocks to
698,029 unique configurations.

\begin{table}[htbp]
\centering
\caption{Results of block aggregation and configuration computation.}
\label{tab:datasets_bloques}
\begin{tabular}{@{}lrrrr@{}}
\toprule
\textbf{Dataset} & \textbf{$|\mathcal{B}|$} & \textbf{Blocks/SKU} & \textbf{Unique configs.} & \textbf{Feasible pairs} \\
\midrule
DS-1 & 11{,}830 & 7.6 & 2{,}145 & 38{,}423\phantom{0}(59.7\%) \\
DS-2 & 198{,}640 & 2.1 & 114{,}869 & 3{,}198{,}156\phantom{0}(92.8\%) \\
DS-3 & 553{,}822 & 1.1 & 352{,}162 & 7{,}885{,}440\phantom{0}(93.3\%) \\
DS-4 & 1{,}725{,}931 & 1.1 & 698{,}029 & 16{,}337{,}817\phantom{0}(97.5\%) \\
\bottomrule
\end{tabular}
\end{table}

The feasible-pair percentage indicates the fraction of (configuration, bin-type) combinations that are geometrically compatible. 
DS-1 has the lowest compatibility rate, 59.7\%, because its high $n_s$ values
produce blocks requiring substantial X-length, which renders many smaller
large-product bin types incompatible. DS-2, DS-3, and DS-4 exhibit much higher
compatibility rates---92.8\%, 93.3\%, and 97.5\%, respectively---as their
smaller blocks fit a broader range of bin types.

All experiments were executed on the Euler HPC cluster at CeMEAI (ICMC/USP)
using the PBS/Torque scheduler with extended walltimes of up to 336~hours.
Each job was allocated 8~CPU cores and 32~GB RAM on nodes equipped with
dual Intel Xeon E5-2650 v2 processors (2.60~GHz, 8~cores per socket).
The software stack consists of Python~3.12.8 and Gurobi~12.0.3.
The SKU limit is set to $m = 4$, as specified by the company based on operational picking productivity requirements. The remaining
algorithm parameters are: BFD residual-space closing threshold of 5\%; column
generation maximum column additions per iteration $P_{\max} = 50$, pricing pool size
$k = 5$, reduced-cost tolerance $\varepsilon = 0.01$, and maximum iterations
$I_{\max} = 1{,}000$.

%-----------------------------------------------------------------------
\subsection{Scalability Experiments on Sub-Instances}
\label{subsec:scalability}
%-----------------------------------------------------------------------

To characterise how computational cost and solution quality scale with instance
size, we construct sub-instances by sampling subsets of SKUs uniformly at
random from each full dataset.
Let $\mathcal{S} \subseteq S$ denote such a
sampled subset of SKUs, so that $|\mathcal{S}|$ is the target size of the
sub-instance.
For each target size~$|\mathcal{S}|$, all items
and blocks of the selected SKUs are included, so the block count $|\mathcal{B}|$
is determined by the aggregation rule of Section~\ref{subsec:aggregation}. Five
independent random seeds are used per size; the full dataset itself is included
as a single-seed reference point. For each sub-instance, BFD is executed first
to obtain a feasible solution, followed by column generation to produce an LP
lower bound. Table~\ref{tab:scalability} reports the aggregated results; all
gaps and runtimes are averaged over the seeds that converged within the time
limit.

\begin{table}[ht]
\centering
\small
\caption{Scalability results ($m=4$). Gap BFD--LP reported as mean $\pm$ std over random seeds. Dashes indicate instances where CG did not converge within the time limit. $^\dagger$~Single seed (full instance).}
\begin{tabular}{@{}llrrrrr@{}}
\toprule
\textbf{Dataset} & \textbf{$|\mathcal{S}|$} & \textbf{$|\mathcal{B}|$} & \textbf{BFD bins} & \textbf{$\eta_{\text{BFD}}$ (\%)} & \textbf{Gap BFD--LP (\%)} & \textbf{CG time (h)} \\
\midrule
\multirow{6}{*}{DS-1} & 50 & 439 & 358 & 70.74 & 1.15 $\pm$ 0.74 & 0.08 \\
 & 100 & 687 & 521 & 71.31 & 2.73 $\pm$ 1.70 & 0.30 \\
 & 250 & 2,093 & 1,527 & 74.30 & 1.60 $\pm$ 0.35 & 1.68 \\
 & 500 & 3,772 & 2,684 & 74.99 & 2.12 $\pm$ 0.54 & 4.52 \\
 & 1,000 & 7,847 & 5,687 & 74.22 & 2.20 $\pm$ 0.29 & 18.68 \\
 & 1,556 & 11,830 & 8,480 & 74.32 & 2.26$^\dagger$ & 65.59 \\
\midrule
\multirow{6}{*}{DS-2} & 100 & 190 & 108 & 57.77 & 24.79 $\pm$ 5.56 & 0.14 \\
 & 500 & 973 & 519 & 58.69 & 22.92 $\pm$ 2.45 & 0.68 \\
 & 2,000 & 4,181 & 2,219 & 58.99 & 22.05 $\pm$ 2.05 & 7.96 \\
 & 10,000 & 20,894 & 10,862 & 59.61 & --- & --- \\
 & 50,000 & 104,843 & 54,553 & 60.34 & --- & --- \\
 & 94,556 & 198,640 & 103,538 & 60.58 & --- & --- \\
\midrule
\multirow{7}{*}{DS-3} & 100 & 108 & 52 & 60.95 & 21.66 $\pm$ 5.89 & 0.05 \\
 & 500 & 554 & 268 & 61.64 & 20.88 $\pm$ 1.70 & 0.39 \\
 & 2,000 & 2,198 & 1,051 & 61.34 & 20.23 $\pm$ 0.93 & 2.41 \\
 & 10,000 & 11,019 & 5,327 & 61.13 & --- & --- \\
 & 50,000 & 54,943 & 26,369 & 61.46 & --- & --- \\
 & 200,000 & 219,841 & 105,538 & 61.48 & --- & --- \\
 & 503,931 & 553,822 & 266,115 & 61.46 & --- & --- \\
\midrule
\multirow{8}{*}{DS-4} & 100 & 115 & 49 & 44.10 & 43.85 $\pm$ 5.61 & 0.07 \\
 & 500 & 568 & 250 & 44.99 & 47.15 $\pm$ 1.85 & 0.49 \\
 & 2,000 & 2,280 & 1,020 & 46.12 & 45.72 $\pm$ 1.48 & 3.78 \\
 & 10,000 & 11,365 & 5,035 & 46.23 & --- & --- \\
 & 50,000 & 56,740 & 25,126 & 46.16 & --- & --- \\
 & 200,000 & 227,089 & 100,578 & 46.19 & --- & --- \\
 & 500,000 & 567,655 & 251,402 & 46.15 & --- & --- \\
 & 1,520,442 & 1,725,931 & 764,622 & 46.13 & --- & --- \\
\bottomrule
\end{tabular}
\label{tab:scalability}
\end{table}

The most salient finding is the divide in BFD solution quality across datasets, which reflects the differences in their operational profiles. For DS-1 the BFD--LP gap is small, 
ranging from 1.15\% to 2.73\% across the tested sizes, with a single-seed value of 2.26\% on the full instance. 
DS-2 and DS-3 form an intermediate group, with gaps around 22--25\% and 20--22\%, respectively, at the sizes where column generation converges. DS-4 is the most challenging dataset, with gaps around 44--47\%. Figure~\ref{fig:gap_skus} illustrates the relation between the  BFD-LP gap and the number of SKUs in the different categories of datasets.

One may also observe in Table~\ref{tab:scalability} that the volumetric efficiency $\eta_{\text{BFD}}$ is essentially flat across sizes within each dataset, indicating stable BFD performance  
with respect to instance scale within each dataset. Across datasets, lower BFD volumetric efficiency is generally associated with
larger BFD--LP gaps: DS-1 combines the highest efficiency with the smallest
gap, whereas DS-4 combines the lowest efficiency with the largest gap. DS-2 and
DS-3 occupy an intermediate region. This outcome validates the rationale for
using structurally diverse datasets: a methodology evaluated only on instances
sharing a common profile would give a misleading picture of its practical
performance.

\begin{figure}[htbp]
    \centering
    \includegraphics[width=0.75\textwidth, trim={0.2cm 0.3cm 0.2cm 0.2cm},
    clip]{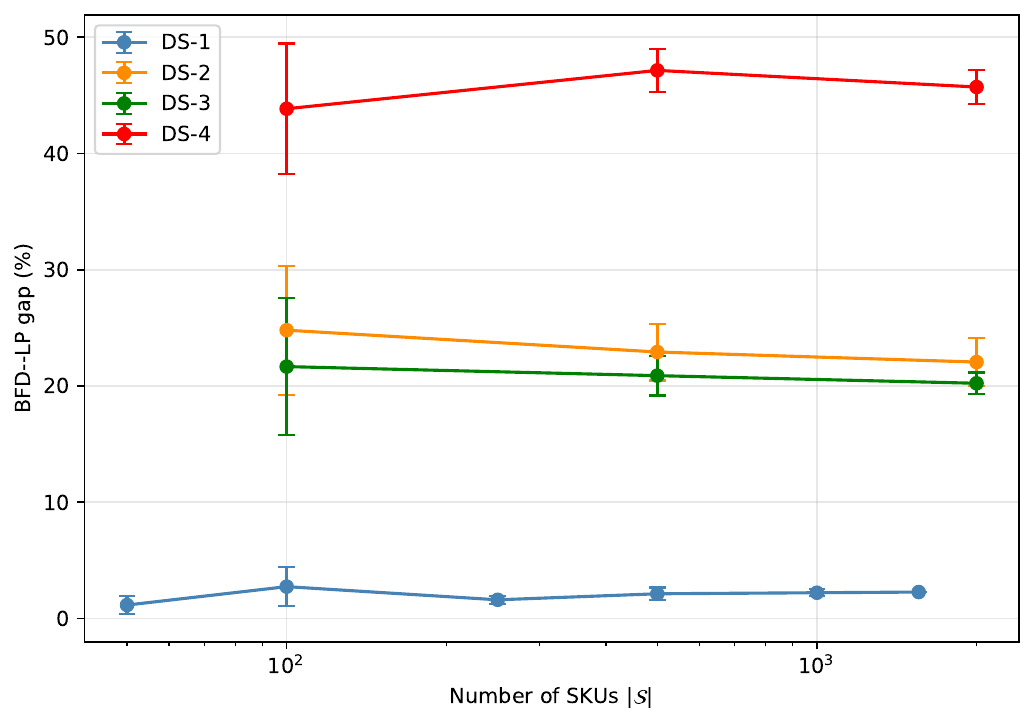}
    \caption{BFD--LP gap (\%) vs. number of SKUs $|\mathcal{S}|$
    (mean $\pm$ std over five seeds, except for the full DS-1 instance, which is a single-seed result). Only sizes where CG converged are shown.}
    \label{fig:gap_skus}
\end{figure}

Figure~\ref{fig:variability_skus} shows the variability of the BFD--LP gap
across random seeds. The coefficient of variation generally decreases as
instances grow. For DS-1 it is high at the smallest sizes, about 64\% at
50 SKUs and 62\% at 100 SKUs, but falls to roughly 13\% at 1,000 SKUs. For
DS-2 it decreases from about 22\% at 100 SKUs to 9\% at 2,000 SKUs; for DS-3
from about 27\% to 5\%; and for DS-4 from about 13\% to 3\%. This confirms that
the gap estimates become increasingly stable at scale.

\begin{figure}[htbp]
    \centering
    \includegraphics[width=0.85\textwidth, trim={0.2cm 0.3cm 0.2cm 0.2cm},
    clip]{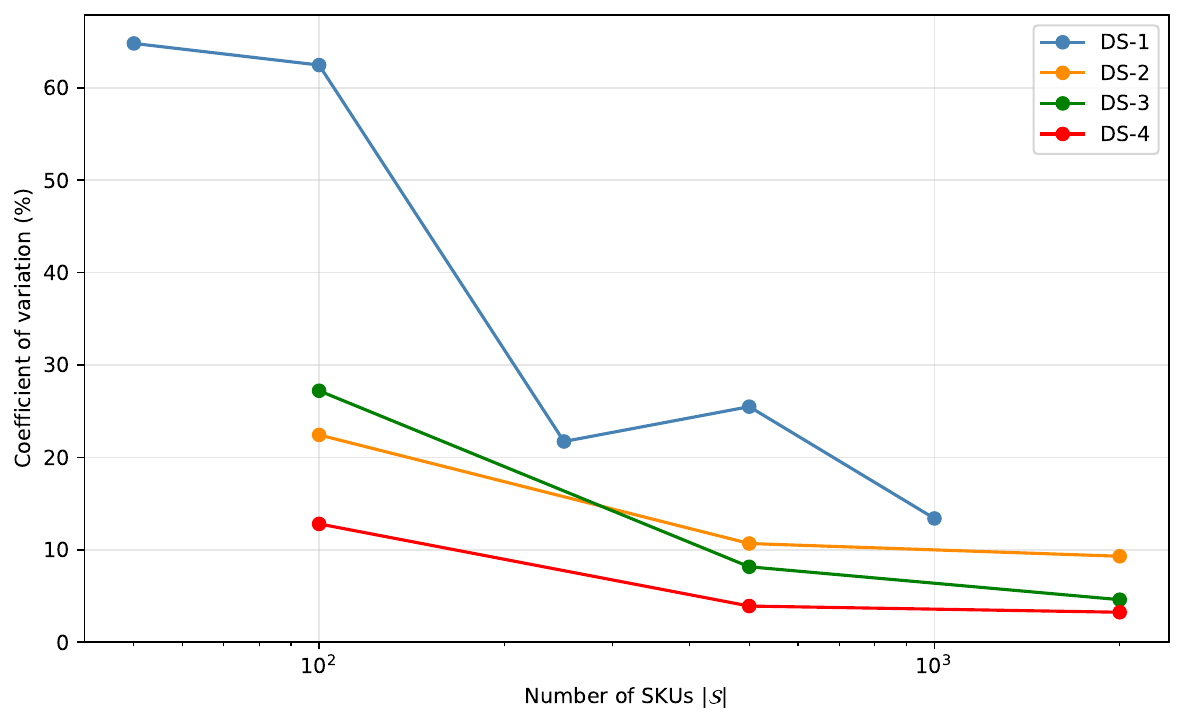}
    \caption{Coefficient of variation (\%) of the BFD--LP gap across random
    seeds, as a function of $|\mathcal{S}|$. Only sizes where CG converged are
    shown.}
    \label{fig:variability_skus}
\end{figure}

Figure~\ref{fig:trivial_skus} compares the number of bins produced by BFD
against the trivial singleton bound. DS-4 still attains the largest reduction,
about 56\%, followed by DS-3 with about 52\% and DS-2 with about 48\%.
DS-1 shows the smallest reduction, about 28\%, because as a large-product
center its items are large relative to the available bins, so fewer blocks can
share a bin.

\begin{figure}[htbp]
    \centering
    \includegraphics[width=0.8\textwidth, trim={0.2cm 0.3cm 0.2cm 0.2cm},
    clip]{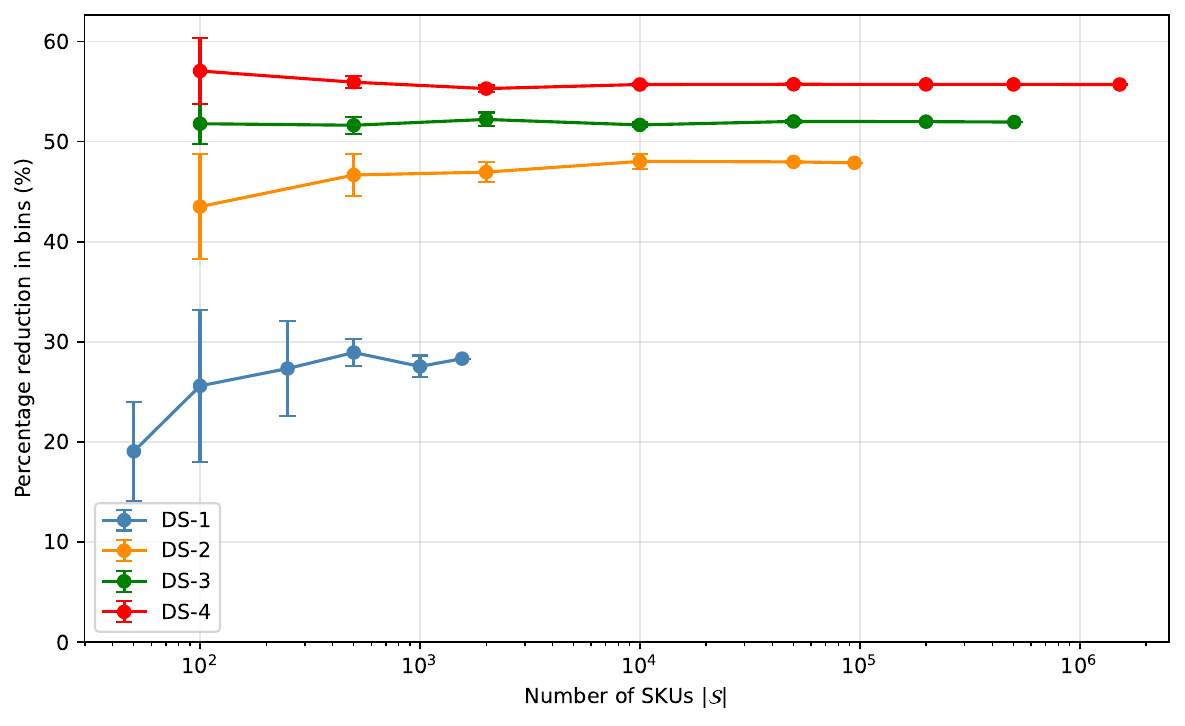}
    \caption{Percentage reduction in the number of bins produced by BFD
    relative to the trivial singleton bound (one bin per block), as a function
    of $|\mathcal{S}|$.}
    \label{fig:trivial_skus}
\end{figure}

Figure~\ref{fig:cgtime_skus} compares the runtime required by CG and BFD.
CG runtime grows steeply with instance size. For DS-1, it grows from about
0.08 hours at $|\mathcal{S}| = 50$ to 65.6 hours at the full instance
($|\mathcal{S}| = 1{,}556$), and converges at every size tested. For DS-2,
DS-3, and DS-4, column generation converges only up to
$|\mathcal{S}| = 2{,}000$, requiring 7.96 hours, 2.41 hours, and 3.78 hours,
respectively. Beyond that, the restricted master problem becomes too large for
the simplex method to resolve within the time budget.
Table~\ref{tab:summary} summarises the convergence frontier and the overall gap statistics.

\begin{table}[ht]
\centering
\small
\caption{Summary by dataset. Max $|\mathcal{S}|$ and Max $|\mathcal{B}|$
indicate the largest instance where column generation converged. Gap BFD--LP is
reported as mean $\pm$ std over all converged tested instances.}
\begin{tabular}{lrrrrrr}
\toprule
\textbf{Dataset} & \textbf{$|\mathcal{S}|$} & \textbf{$|\mathcal{B}|$} & \textbf{Max $|\mathcal{S}|$} & \textbf{Max $|\mathcal{B}|$} & \textbf{Gap BFD--LP (\%)} & \textbf{$\eta_{\text{BFD}}$ (\%)} \\
\midrule
DS-1 & 1{,}556 & 11{,}830 & 1{,}556 & 11{,}830 & 1.97 $\pm$ 0.96 & 73.16 \\
DS-2 & 94{,}556 & 198{,}640 & 2{,}000 & 4{,}181 & 23.25 $\pm$ 3.63 & 59.14 \\
DS-3 & 503{,}931 & 553{,}822 & 2{,}000 & 2{,}198 & 20.92 $\pm$ 3.37 & 61.34 \\
DS-4 & 1{,}520{,}442 & 1{,}725{,}931 & 2{,}000 & 2{,}280 & 45.57 $\pm$ 3.54 & 45.72 \\
\bottomrule
\end{tabular}
\label{tab:summary}
\end{table}

DS-1 is the only dataset for which CG converges at its full instance
(11,830 blocks, 65.6 hours): its complete inventory is small enough to remain
below the convergence frontier. For DS-2, DS-3, and DS-4, the full instances are
two to three orders of magnitude larger ($10^5$--$10^6$ blocks), so convergence
is limited to roughly $|\mathcal{S}| = 2{,}000$. The BFD runtime remains orders of
magnitude below CG at every size (Figure~\ref{fig:cgtime_skus}), confirming BFD as
the only feasible method for the full-scale datasets beyond DS-1.

\begin{figure}[htbp]
    \centering
    \includegraphics[width=0.8\textwidth, trim={0.2cm 0.3cm 0.2cm 0.2cm},
    clip]{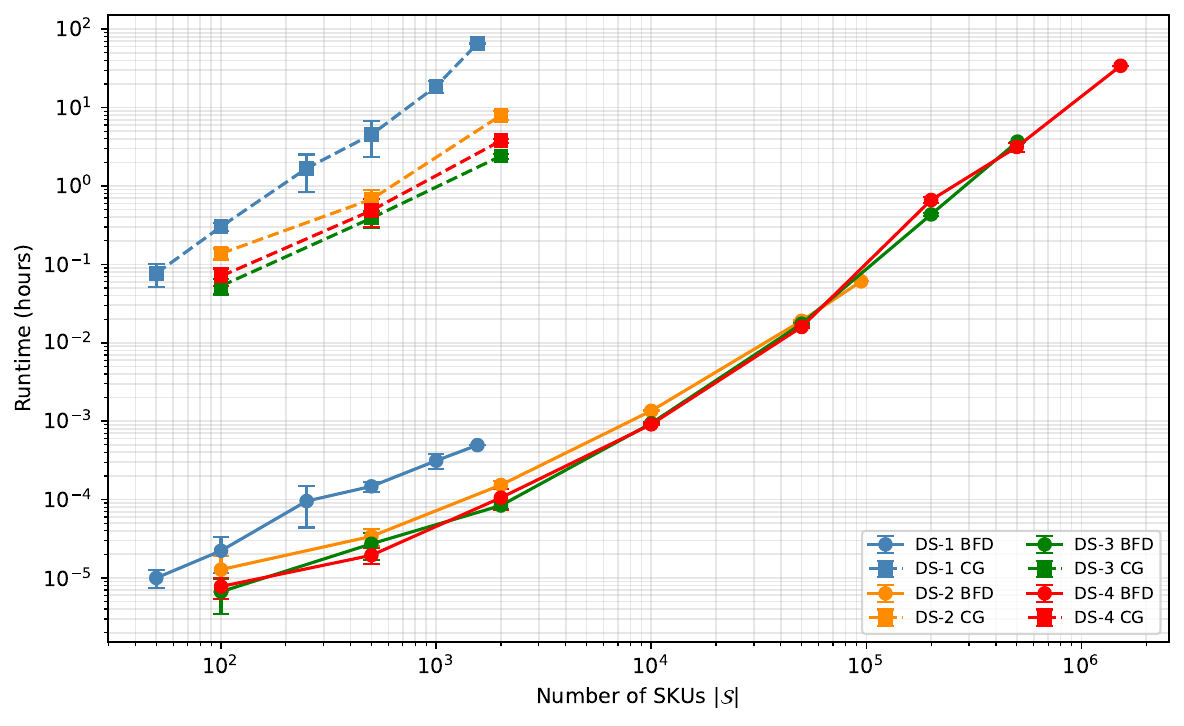}
    \caption{BFD runtime (solid lines, circles) and CG runtime (dashed lines,
    squares) in hours (log scale) vs.\ number of SKUs $|\mathcal{S}|$ (log
    scale). CG curves are shown only where convergence was achieved.}
    \label{fig:cgtime_skus}
\end{figure}

%-----------------------------------------------------------------------
\subsection{Full-Scale Results}
\label{subsec:fullscale}
%-----------------------------------------------------------------------

This subsection applies the complete methodology to the four full-scale synthetic datasets.
The BFD heuristic is used to obtain feasible solutions, and column
generation is applied to certify solution quality via LP lower bounds. The RMP
is initialized with the patterns from the BFD solution.

%Table~\ref{tab:fullscale_results} consolidates the main results across all datasets.
Table~\ref{tab:fullscale_results} consolidates the BFD results across all datasets. The BFD heuristic scales effectively to industrial-size instances. Its
runtime ranges from about 2 seconds on DS-1 to %15.7
34.1 hours on DS-4 (1.7M~blocks); the dominant cost is the best-fit search, which for each block scans all open bins of compatible types and grows more expensive as the number of simultaneously
open bins increases.

\begin{table}[htbp]
\centering
\caption{Full-scale BFD results ($m = 4$).}
\label{tab:fullscale_results}
\begin{tabular}{@{}lrrrrr@{}}
\toprule
\textbf{Dataset} & \textbf{$|\mathcal{B}|$} & \textbf{Isolated} & \textbf{BFD bins} & \textbf{BFD vol.\ (cm$^3$)} & \textbf{BFD time} \\
\midrule
DS-1 & 11{,}830 & 1.96\% & 8{,}480 & 14{,}486{,}827{,}800 & 2 s \\
DS-2 & 198{,}640 & 0.02\% & 103{,}538 & 35{,}545{,}032{,}600 & 3.6 min \\
DS-3 & 553{,}822 & 0.00\% & 266{,}115 & 11{,}735{,}936{,}550 & 3.7 h \\
DS-4 & 1{,}725{,}931 & 0.00\% & 764{,}622 & 25{,}839{,}967{,}500 & 34.1 h \\
\bottomrule
\end{tabular}
\end{table}

The isolated-block pre-filter has limited effect: it identifies 1.96\% of
blocks as isolated for DS-1, 0.02\% for DS-2, and essentially none for DS-3 and
DS-4. Thus, almost all blocks remain candidates for shared-bin packing in the
subsequent BFD phase.

Table~\ref{tab:bfd_packing} presents the percentage of bins per dataset that share 1 to 4 blocks, besides their volumetric efficiency. DS-1 attains the highest volumetric efficiency, 74.3\%. The bins of DS-1 exhibit
moderate sharing: 63.1\% hold one block and 36.9\% hold two or more blocks, consistent with the approximately 28\% reduction
over the singleton bound.

\begin{table}[htbp]
\centering
\caption{BFD packing structure and volumetric efficiency ($m = 4$).}
\label{tab:bfd_packing}
\begin{tabular}{@{}lrrrrr@{}}
\toprule
 & & \multicolumn{4}{c}{\textbf{Percentage of bins}} \\
\cmidrule(l){3-6}
\textbf{Dataset} & \textbf{$\eta_{\text{BFD}}$} & \textbf{1 block} & \textbf{2 blocks} & \textbf{3 blocks} & \textbf{4 blocks} \\
\midrule
DS-1 & 74.3\% & 63.1\% & 34.9\% & 1.3\% & 0.7\% \\
DS-2 & 60.6\% & 27.0\% & 61.6\% & 4.0\% & 7.4\% \\
DS-3 & 61.5\% & 18.9\% & 65.0\% & 5.1\% & 11.0\% \\
DS-4 & 46.1\% & 12.9\% & 62.8\% & 10.0\% & 14.3\% \\
\bottomrule
\end{tabular}
\end{table}

DS-2 achieves 60.6\% efficiency. A large majority of its bins, 73.0\%, contain
two or more blocks, with an average of approximately 1.92 blocks per bin.

DS-3 achieves 61.5\% efficiency. Its packing remains dense in terms of block
sharing: 81.1\% of bins hold two or more blocks, and the solution attains about
52\% bin reduction relative to the singleton bound.

DS-4 is the largest instance, with 1{,}725{,}931 blocks, and achieves 46.1\%
efficiency. Most bins, 87.1\%, contain two or more blocks, and 14.3\% reach the
maximum allowed number of SKUs per bin. The median $n_s = 1$ means that most
SKUs generate a single block, making the problem structurally closer to a
classical variable-sized bin packing problem, but the low volumetric efficiency
and large BFD--LP gaps on sub-instances indicate that this dataset is the most challenging for the heuristic.

%\GG{
%These efficiency levels are not directly comparable with those in the assortment-problem literature. \citet{alonso2016determining} equate shipper cost with volume, making their waste rates of 5.79\% and 9.27\% the complement of our $\eta$. However, they freely dimension a few shipper types for a few product types under fundamentally different constraints: their four-to-twelve limit concerns total units, not distinct product types, and their \emph{this-side-up} rule constrains orientation, not relative placement (Section~\ref{sec:related_works}). In contrast, we accommodate an entire inventory using a fixed catalog, with $m = 4$, per-SKU quantity limits, and lateral-only placement. Their ESICUP instances lack the parameters $T$, $m$, and $n_s$ that define our problem.}

Table~\ref{tab:cg_results} shows a summary of the column generation results per dataset. One may observe that the column generation converged only for DS-1, certifying a BFD--LP gap of
2.26\%. The LP lower bound is 14{,}166{,}101{,}333~cm$^3$, compared with a
BFD solution value of 14{,}486{,}827{,}800~cm$^3$. The resulting column set
contains 8{,}347 patterns. The subsequent RIP, solved over this column set,
yielded an integer solution of 14{,}178{,}468{,}000~cm$^3$ using 7{,}381 bins.
The LP--IP gap is 0.087\%, the final MIP gap is 0.075\%, and the BFD--IP gap is 2.17\%, indicating that most of the remaining
difference is due to the BFD heuristic rather than to the integrality gap of
the restricted master problem. %These results confirm that the set partitioning relaxation is very tight for DS-1; most of the remaining difference comes from the BFD heuristic rather than from the integrality gap of the restricted master problem.
These results confirm that the set partitioning relaxation is tight for DS-1; whether the same holds for the larger datasets remains open, since column generation did not converge for them within the time budget.

\begin{table}[htbp]
\centering
\caption{Column generation results ($m = 4$).}
\label{tab:cg_results}
\begin{tabular}{@{}lrrr@{}}
\toprule
\textbf{Dataset} & \textbf{$z^*_{\text{LP}}$ (cm$^3$)} & \textbf{BFD--LP gap} & \textbf{CG time} \\
\midrule
DS-1 & 14{,}166{,}101{,}333 & 2.26\% & 65.6 h \\
DS-2 & --- & --- & $>$5 days$^\dagger$ \\
DS-3 & --- & --- & $>$5 days$^\dagger$ \\
DS-4 & --- & --- & $>$5 days$^\dagger$ \\
\bottomrule
\end{tabular}
\par\smallskip
{\small $^\dagger$Column generation interrupted at the 5-day wall-clock limit during the first RMP solve, before any pricing iteration completed.}
\end{table}

For DS-2, DS-3, and DS-4, column generation was interrupted at the 5-day
wall-clock limit during the first RMP solve, before any pricing iteration could
be completed. The bottleneck is the simplex resolution of the initial RMP: with
$|\mathcal{B}| > 10^5$ covering constraints and $|\mathcal{P}'_0| > 10^5$
columns, where $\mathcal{P}'_0$ denotes the initial pattern set generated from
the BFD solution, the first LP solve alone exceeds the available computation
time. This limitation is consistent with the scalability experiments, which show that the CG convergence frontier lies at approximately $|\mathcal{S}| = 2{,}000$
($|\mathcal{B}| \approx 2{,}000$--$4{,}000$) for these three datasets. This limitation is computational rather than methodological: scaling column
generation to these full-scale instances would require reducing the effective
size of the master problem, for example through constraint aggregation, master
problem decomposition, or stronger warm-start strategies.

\section{Conclusions}
\label{sec:conclusions}

This paper introduces and studies the bin dimensioning problem in e-commerce
fulfillment centers, a large-scale variant of the assortment problem motivated by the company's real operational needs.
The problem asks to select bin types from a discrete catalog and assign all items to bins so as to minimize total
bin volume, subject to a limit on the number of distinct SKUs per bin, a
per-SKU quantity limit, and a stacking rule that restricts items of different
SKUs to side-by-side placement along the bin length.

The central methodological contribution is a dimensional decomposition that
reduces the original three-dimensional packing problem to a one-dimensional
block-positioning problem. This decomposition exploits the interplay between the
stacking constraint and the identity of items within each SKU: items of the same
SKU are aggregated into blocks, the stacking rule eliminates the Y and Z
coordinates from the non-overlap constraints between blocks of distinct SKUs,
and the optimal X-dimension of each block is computed in $O(1)$ time by
maximizing the packing density in the Y--Z cross-section. The reduction is what
makes the problem tractable at industrial scale, transforming a
three-dimensional geometric packing problem with quadratic non-overlap
constraints into a one-dimensional capacity problem amenable to classical bin
packing techniques.

We proposed three solution approaches built on this decomposition: a compact
MILP formulation that serves as a formal problem statement, a scalable
Best-Fit-Decreasing heuristic with isolated-block pre-filter and round-robin
ordering that produces 
feasible solutions for the largest
instances, and a column generation scheme with per-type pricing subproblems
that yields tight LP lower bounds.

%On the four synthetic datasets generated from proprietary real-world data, the BFD heuristic scaled effectively to instances with up to 1.7~million blocks, with execution times ranging from a few seconds to 34.1~hours.  Scalability experiments on sub-instances of increasing size show that, for DS-2, DS-3, and DS-4, column generation converges up to approximately $|\mathcal{S}| = 2{,}000$ SKUs, with runtimes up to 7.96~hours. For DS-1, column generation converges at all tested sizes, including the full instance with 11{,}830 blocks and a runtime of 65.6~hours. On the full DS-1 instance, column generation certified a BFD--LP gap of 2.26\%. The subsequent RIP solved over the generated column set produced an integer solution with an LP--IP gap of only 0.087\%, confirming the tightness of the set partitioning relaxation for this dataset. For the three larger full-scale datasets, the initial RMP solve exceeded the 5-day wall-clock limit, consistent with the scalability frontier. The experiments also reveal a structural divide in solution quality: the BFD--LP gap is about 2\% for DS-1, around 21--23\% for DS-2 and DS-3, and about 46\% for DS-4.

We validated this methodology on four synthetic datasets calibrated on
proprietary real-world data provided by the company, deliberately selected to
span both the scale and the profile diversity of its network: from 1{,}556 to
1{,}520{,}442 SKUs, or equivalently, from 11{,}830 to 1{,}725{,}931 blocks and
from 0.3 to 4.0 million individual items across large- and small-product
fulfillment centers. The BFD heuristic solves all four datasets at full scale
on a single 8-core node with 32~GB of RAM, with runtimes ranging from
2~seconds on DS-1 to 34.1~hours on DS-4, the largest instance, with
1.7~million blocks. The entire inventory of an industrial fulfillment center
is therefore dimensioned in a single run, at instance sizes several orders of
magnitude beyond those considered in the assortment literature
(Section~\ref{sec:related_works}).

Column generation supplies the certification that makes these solutions
verifiable rather than merely fast. For DS-1 it converges at every tested
size, including the full instance (11{,}830 blocks, 65.6~hours), where it
certifies a BFD--LP gap of 2.26\%; the RIP solved over the generated column
set returns an integer solution within 0.087\% of the LP bound, showing that
the residual difference is attributable to the heuristic rather than to the
integrality gap of the set partitioning relaxation. For DS-2, DS-3, and DS-4,
convergence is attained on sub-instances of up to approximately
$|\mathcal{S}| = 2{,}000$ SKUs, with runtimes of up to 7.96~hours, whereas at
full scale the initial RMP solve exceeds the 5-day wall-clock limit. This
frontier is computational rather than methodological, and the two methods
complement each other: BFD delivers solutions at industrial scale, and column
generation quantifies how much volume they leave on the table.

The experiments further reveal a structural divide in solution quality,
driven by the interaction between the operational constraints and the product
profile of each facility. Where column generation converges, the BFD--LP gap
is about 2\% for DS-1, 21--23\% for DS-2 and DS-3, and about 46\% for DS-4,
and volumetric efficiency moves in the opposite direction, from 74.3\% on
DS-1 down to 46.1\% on DS-4.This divide is not an artifact of instance size: within each dataset, efficiency remains essentially constant across the full range of tested instance sizes. Rather, it is driven by the product profile. DS-4 combines the smallest items with the highest proportion of unit-demand SKUs (57.80\%). Even with a bin catalog specifically designed for its product category, the SKU-per-bin limit and thelateral-only placement rule leave more than half of the available bin volume unused.
This finding has a direct practical reading: bin catalogs, and the operational
rules that constrain them, should be dimensioned per facility profile rather
than uniformly across the network, and the facilities whose achievable
utilization is lowest are precisely those where a method stronger than BFD has
the most volume to recover.

Several directions for future work emerge from these results. On the
methodological side, constraint aggregation or decomposition strategies for
the master problem could extend column generation to the larger datasets where
the current approach is limited by the size of the initial RMP; 
a branch-and-price algorithm would be the natural next step toward closing the
optimality gap exactly, rather than relying on the restricted integer program.
On the application side, this work constitutes the first stage of a broader project with the company. Having determined the catalog of bin types required to accommodate peak-demand inventories such as those arising during Black Friday, the subsequent stages
address two further problems: how to distribute the selected bins across the
physical shelving structure of each fulfillment center, and how to assign
items to specific bin positions so as to minimize the time and effort required
by warehouse operators during put-away and picking operations. Together, these
three stages form a complete storage design pipeline, from bin sizing to rack
layout to item placement.

\section*{Acknowledgments}
We acknowledge financial support from FAPESP 
(grants~2026/00417-9,  2022/05803-3) and  Conselho Nacional de Desenvolvimento Científico e Tecnológico (CNPq) (grants 305157/2025-6, 403735/2021-1). The computational experiments were carried 
out using the computational resources of the Center for 
Mathematical Sciences Applied to Industry (CeMEAI), funded by 
FAPESP (grant~2013/07375-0).

%==============================================================================
% REFERENCES
%==============================================================================

\bibliographystyle{apalike}
\bibliography{references}

%==============================================================================
% APPENDIX
%==============================================================================

\end{document}